\documentclass[final,12pt]{colt2025} % Include author names

\title[Visual tests using several safe confidence intervals]{Visual tests using several safe confidence intervals}
\usepackage{times}
\usepackage{{booktabs}}
\usepackage{tikz}
\usetikzlibrary{calc,shapes.arrows,decorations.pathreplacing}
\usetikzlibrary{arrows.meta, patterns.meta}
\tikzset{
  @numline/@lbracket/.style 2 args={insert path={  ([shift={(#1)}]  1.5pt,  4pt) coordinate (area_start) -| (#1) |- +( 1.5pt, -4pt)}},
  @numline/@rbracket/.style 2 args={insert path={#2([shift={(#1)}] -1.5pt, -4pt)                         -| (#1) |- +(-1.5pt,  4pt)}},
  /utils/create to paths for combo/.style args={#1_#2}{
    #1-#2/.style={to path={[insert path/.expanded={
      edge[path only, numline delimiters, @numline/@#1_#2=]  (\tikztotarget)}]}},
    @numline/@#1_#2/.style={to path={[@numline/@l#1={\tikztostart}{},
                                      @numline/@r#2={\tikztotarget}{##1}] ##1 (area_start)}}
},
  /utils/create to paths for combo/.list={
    angle_bracket, angle_angle,  bracket_angle, bracket_bracket,
         _bracket, angle_,              _angle, bracket_}}
\tikzset{
  numline delimiters/.style={thick, draw},
  numline area/.style={pattern color=gray, thin, pattern={Lines[angle=90]}}}

\usepackage{todonotes}

\usepackage{caption}
\usepackage{bbm}
\usepackage{multirow}

\renewcommand{\P}{\mathbb{P}}
\newcommand{\E}{\mathbb{E}}
\newcommand{\R}{\mathbb{R}}
\newcommand{\N}{\mathbb{N}}

\renewcommand{\d}{\mathrm{d}}

\coltauthor{%
 \Name{Timothée Mathieu} \Email{timothee.mathieu@inria.fr}\\
 \addr Inria, Université de Lille, CNRS, Centrale Lille, UMR 9189 – CRIStAL
}

\begin{document}

\maketitle

\begin{abstract}%
We propose a new statistical hypothesis testing framework which decides visually, using confidence intervals, whether the means of two samples are equal or if one is larger than the other. With our method, the user can at the same time visualize the confidence region of the means and do a test to decide if the means of the two populations are significantly different or not by looking whether the two confidence intervals overlap. To design this test we use confidence intervals constructed using e-variables, which provide a measure of evidence in hypothesis testing. We propose both a sequential test and a non-sequential test based on the overlap of confidence intervals and for each of these tests we give finite-time error bounds on the probabilities of error. We also illustrate the practicality of our method by applying it to the comparison of sequential learning algorithms.
\end{abstract}

\begin{keywords}%
Confidence intervals, hypothesis testing, E-values, Anytime confidence intervals%
\end{keywords}

\section{Introduction}

Confidence intervals on the mean are an important tool in statistics and are often used by practitioners to compare the means of several populations ``by eye'', i.e. deciding that the mean of two populations are different if and only if the confidence intervals are disjoint.
In theory, this comparison ought to be a test of hypothesis between the means, but deducing a pairwise test from confidence intervals on individual series of data can be tricky, this is the subject of this paper. 

We see two natural usages of such confidence intervals. One is in exploratory data analysis, to visually decide whether a certain feature is relevant to decide between two categories in a classification problem. A second usage is when comparing several machine learning algorithms. The results of such comparison can be several confidence intervals on the performances, one for each algorithm, and the plot of these confidence intervals can then be used to see quickly which of two algorithms is the best one and get information on the actual mean performance of each algorithm.

If we have two samples $X_1,\dots,X_n$ i.i.d. from a law $P$ and $Y_1,\dots,Y_m$ i.i.d. from a law $Q$, used to construct two confidence intervals $C_n(\alpha; X)$ and $C_m(\alpha;Y)$ such that $\P(\E_P[X] \in C_n(\alpha; X)) \ge 1-\alpha$ and $\P(\E_Q[Y] \in C_m(\alpha;Y)) \ge 1-\alpha$ for some $\alpha \in (0,1)$. We are then tempted to say that if the two confidence intervals $C_n(\alpha;X)$ and $C_m(\alpha;Y)$ do not overlap, then we can reject $H_0:\E_P[X] = \E_Q[Y]$ otherwise we do not reject $H_0$. We call this procedure a test of overlap. It is easy to see that such a test has a type I error at most $2\alpha$ however it is known that the actual error of such test is often lower than $2\alpha$, and this asks the question of what is the actual error of this test, and what can we say about its type II error. 

In the case of Gaussian random variables, \cite{goldstein1995graphical} give a bound on the type I error of a test of overlap  provided that the variances of $P$ and $Q$ are known. \textbf{Our main contribution} is to provide confidence intervals and associated \textbf{visual tests} of overlap. We bound all the probabilities of error of our test giving \textbf{theoretical guarantees in a nonparametric setting} in which we suppose that the distributions are bounded. We construct the confidence intervals using the betting principle~\citep{waudby2024estimating} and the notion of e-variables which can be seen as a notion of evidence: the higher the e-variable the less likely it is that $H_0$ is true. It is convenient to use e-variables because they can easily be combined and provide a natural interpretation as evidence in a test. We propose both anytime and fixed time results allowing for a sequential test in which the data come as a stream.

More precisely, in the sequential setting, let $X_1,X_2,\dots$ be i.i.d. from a law $P$  and $Y_1,Y_2,\dots$ be i.i.d. from a law $Q$ independents of $X_1,X_2,\dots$, with $P$ and $Q$ having bounded supports. With these samples, we construct sequences of anytime confidence intervals $\{C_n(\alpha;X)\}_n$ and $\{C_m(\alpha,Y)\}_m$. In the non-sequential setting, we have access to a sample $X_1,\dots,X_n$ and a sample $Y_1,\dots,Y_n$ from which we construct two confidence intervals $C_n(\alpha;X)$ and $C_m(\alpha,Y)$.
Then, using these confidence intervals, we design a sequential and a non-sequential test of visual overlap to decide between the following three hypotheses
$$H_1^-: \E_P[X] < \E_Q[Y] - \Delta,\quad H_0 : \E_P[X] = \E_Q[Y], \quad H_1^+: \E_P[X] > \E_Q[Y] + \Delta,$$
for some effect size $\Delta>0$ assumed fixed in advance. We show error bounds that are summarized in Table~\ref{table:errors} (see the Theorems~\ref{th:typeIerror},~\ref{th:typeII_anytime},~\ref{th:typeII_hoeffding},~\ref{th:typeIII} for the non-asymptotic error bounds). Remark that the three hypotheses can easily be reduced to the more usual bilateral test if needed.
\begin{figure}[h]
\begin{center}
\small
\begin{tabular}{ |p{2cm}||p{1.5cm}|p{1.5cm}|p{1.5cm}|p{1.5cm}|p{1.5cm}|p{1.5cm}|  }
 \hline
 &\multicolumn{2}{|c|}{$\E_P[X] < \E_Q[Y] - \Delta$} & \multicolumn{2}{|c|}{$\E_P[X] = \E_Q[Y]$} & \multicolumn{2}{|c|}{$\E_P[X] > \E_Q[Y] +\Delta$} \\
 \hline
Sequential & Yes & No & Yes & No & Yes & No  \\
 \hline
\hline
Decide $H_1^-$    &  No error   &  No error   & $ \alpha$ & $ \alpha e^{-\Delta\Omega(\sqrt{n})}$& $2\alpha^2$ & $2\alpha^2$\\
 \hline
Decide $H_0$    &  $\alpha^{3/2}$   &  $\alpha^{3/2}$  & No error& No error& $\alpha^{3/2}$ & $\alpha^{3/2}$\\
 \hline
Decide $H_1^+$    &  $2\alpha^2$   &  $2\alpha^2$   & $ \alpha$ & $ \alpha e^{-\Delta\Omega(\sqrt{n})}$& No error & No error\\
 \hline
\end{tabular}
\end{center}
\caption{Approximate probabilities of error of the overlap test for $n$ sufficiently large.\label{table:errors}}
\end{figure}

Finally, in Section~\ref{sec:xp} we showcase the practical efficiency of our tests and their usage to compare the performance of some reinforcement learning and bandit algorithms.

\paragraph{Related works on tests from confidence intervals}
Previous studies have primarily focused on the Gaussian setting with known variance. The first work on this problem dates back to~\cite{goldstein1995graphical}, after which several applied papers~\citep{AUSTIN2002194,knol2011mis,schenker2001judging} have either proposed alternative methods or commentary on how to do visual tests using confidence intervals, but always in the case of a Gaussian model with known variances.

\paragraph{Related works on nonparametric confidence intervals using E-variables}
E-values~\citep{10.1093/jrsssb/qkae011,ramdas2023game,ramdas2024hypothesis} are an alternative to p-values to perform hypothesis testing. The bigger the e-value is, the more confident we are that we should reject $H_0$, which means that e-values are a measure of \textbf{evidence}. Moreover, e-values are often \textbf{anytime-valid}, in other words they can deal with data that are collected sequentially. Of particular interest for our article, we will look at the technique leveraging e-values to construct confidence intervals introduced in \cite{waudby2024estimating} in the case of bounded random variables. 

\section{Construction of confidence intervals and test of overlap}\label{sec:evalue}
In this section we introduce the confidence intervals and tests that are the subject of our analysis.

\subsection{Construction of confidence intervals using e-variables}
First, we present briefly what are e-variables and how to use them to construct confidence intervals. For a more comprehensive  background on this subject, refer to~\cite{waudby2024estimating}.

\paragraph{E-variables}
Given a test $H_0$ versus $H_1$, an e-variable is a non-negative random variable $E_n$ constructed from a sample $X_1,\dots,X_n$ following a law $P$ such that if $P \in H_0$ then $\E[E_n] \le 1$, the observed value of an e-variables is called an e-value. Using Markov inequality, we have that $\P(E_n \ge 1/\alpha)\le  \alpha \E[E_n] \le \alpha$ hence e-variables are a way to define tests with controlled type I error. We can often get an even stronger result, if $E_n$ is a super-martingale, using Ville's inequality to conclude to an anytime result: $\P(\exists n \ge 1 \mid \ E_n \ge 1/\alpha)\le  \alpha$. In addition to type I error, we also want to have $E_n$ as large as possible under $H_1$, which will mean having high power for the test.

For our purpose, we consider $E_n(m;X,W)$ the e-variables (inspired by~\cite[Section 4.4]{waudby2024estimating}) defined by 
\begin{equation}\label{eq:def_En} 
E_n(z;X,W) = \frac{1}{2}\max\left(\prod_{t=1}^n (1+W_t(X)(X_t -z)), \prod_{t=1}^n (1-W_t(X)(X_t -z))\right),
\end{equation}
where $W_1(X),\dots,W_n(X)$ are positive random variables. We suppose that each $W_t(X)$ is a positive random variable that is measurable with respect to the filtration spanned by $X_1,\dots,X_{t-1}$ (i.e. $W_t$ depends only on $X_1,\dots,X_{t-1}$) and is such that $W_t(X)|x -m|$ is smaller than $1$ for all $x$ in the support of $P$. It is easy to see that $\E_P[E_n(\E_P[X])]\le 1$ which means that $E_n(z;X,W)$ is an e-variable for the test $H_0:\E_P[X] = z$ versus $H_1:\E_P[X] \neq z$.

\paragraph{Confidence intervals from e-variables}

Using the function $E_n$, we can construct a the set $C_n(\alpha;X,W)$ by
$$C_n(\alpha;X,W) = \{ z \in \R : \ E_n(z;X,W) \le 1/\alpha\}.$$
Using Markov's inequality and $\E_P[E_n(\E_P[X])]\le 1$, this leads to $\P\left(\E_P[X] \in C_n(\alpha;X)\right) = \P\left(E_n(\E_P[X]) \le 1/\alpha\right) \le \alpha$ and because $E_n$ is convex (see Lemma~\ref{lem:convex}), $C_n(\alpha;X,W)$ is indeed an interval.

Remark that $C_n(\alpha;X,W)$ is in fact an anytime confidence interval of level $1-\alpha$, because it can be seen that $E_n(\E_P[X];X,W)$ is a non-negative super-martingale, and we can use Ville's inequality instead of Markov inequality to conclude that 
$$\P\left(\forall n \ge 1 \mid \E_P[X] \in C_n(\alpha;X)\right) \ge 1-\alpha.$$ 
The precise form of the weight sequence $\{W_t\}_t$ that we will consider is a trade-off between the length of the confidence intervals (see Section~\ref{sec:power} and discussions on predictable plugin in \cite{waudby2024estimating}) and the probability of having an overlap between two intervals when the means are the same (see Section~\ref{sec:typeI}).

\begin{remark}[Notations]
In what follows, we will use the shorthand $E_n(z; X)$ and $C_n(\alpha;X)$ to mean $E_n(z; X,W)$ and $C_n(\alpha;X,W)$ respectively when the weight sequence considered is obvious in context. For an interval $I$, we denote $L(I)=\max_{x \in I} x - \min_{y \in I} y$ for the length of the interval $I$. We denote $\ I \ge a\ $  if $\ \min_{x \in I} x \ge a\ $ and $\ I \ge J + a\ $ if $\ \min_{x \in I} x \ge \max_{y\in J} y + a\ $.
\end{remark}
\subsection{Tests of overlap}

Let $P$, $Q$ be two distributions with bounded supports. Inspired by works on directional hypotheses tests~\citep{Leventhal-1996}, we define the following three hypotheses:
$$H_1^-: \E_P[X] < \E_Q[Y] - \Delta, \quad H_0 : \E_P[X] = \E_Q[Y], \quad H_1^+: \E_P[X] > \E_Q[Y] + \Delta.$$
We are considering two settings: sequential test in which the data come in stream and fixed-time (non-sequential) test.

\paragraph{Fixed-time test} Suppose $C_n(\alpha, X, W)$ and $C_n(\alpha, Y, W)$ are fixed time confidence intervals of $\E_P[X]$ and $\E_Q[Y]$. The test of overlap is to decide on $H_1^+$ if $C_n(\alpha, X, W) > C_m(\alpha, Y, W)$, decide on $H_1^-$ if $C_n(\alpha, X, W) < C_m(\alpha, Y, W)$ and decide on $H_0$ if $C_n(\alpha, X, W) \cap C_m(\alpha, Y, W)\neq \emptyset$. See Equation~\eqref{eq:weight_hoeffding_fixed_time} for the actual weights that will be used for this case. 

\paragraph{Sequential test} On the other hand, in the case of sequential test, we cannot just consider the three decisions used in the fixed-time test because when there are not enough samples, the confidence intervals are very likely to overlap and the test would almost always decide on $H_0$ early and have a very low power. Instead, we use the test described in Algorithm~\ref{algo:anytime} that depends on the effect size that we hope to detect. Remark that the choice of which distribution to sample at each iteration on line 4 of Algorithm~\ref{algo:anytime} can depend on the past. The only requirement is that we cannot sample from $P$ and $Q$ and then decide which to keep according to what was sampled: everything observed must be included in the test. See Equation~\eqref{eq:lambda_any_time} for the actual weights that will be used for this case. 

\begin{algorithm2e}[h]
\caption{Anytime test of overlap}\label{algo:anytime}
\vspace{0.5em}
\begin{enumerate*}
\item Collect $t_0\ge 0$ samples $X_1,\dots,X_{t_0}$ from $P$ and $Y_1,\dots,Y_{t_0}$ from $Q$.
\item While no decision is taken, do
\begin{enumerate*}
\item Collect a sample either from $P$ or from $Q$.
\item If $C_n(\alpha;X,W) > C_m(\alpha;Y,W) $, decide on $H_1^+$.
\item If $C_n(\alpha;X,W) < C_m(\alpha;Y,W) $, decide on $H_1^-$.
\item If $C_n(\alpha;X,W) \cap C_m(\alpha;Y,W) \neq \emptyset$ and $L(C_n(\alpha;X,W) \cup C_m(\alpha;Y,W))\le \Delta$, accept $H_0$.
\item Else, continue to collect.
\end{enumerate*}
\end{enumerate*}
\vspace{-0.5em}
\end{algorithm2e}

\begin{remark}[Burn-in period]
In Algorithm~\ref{algo:anytime} we suppose that $t_0$ samples are collected initially from each probability, this is a burn-in period that will greatly improve the type I error bound that we obtain and although one can still apply our result with $t_0 = 1$ we will see in the experimental section that allowing for a first non-sequential phase greatly improve our error bounds.
\end{remark}

\begin{remark}[directional versus bilateral test]
Instead of the three hypotheses $H_1^-, H_0$ and $H_1^+$, we could also use the more usual bilateral test $H_0:\E_P[X] = \E_Q[Y]$ and $H_1:|\E_P[X] - \E_Q[Y]|>\Delta$. However, we prefer to explicitly consider three distinct hypotheses as it seems more useful in practice to decide which of $P$ or $Q$ has a largest mean. It is possible to modify our algorithm to rejects $H_0$ whenever the confidence intervals do not intersect and this answers the problem of bilateral test. Our theoretical analysis can then be used to bound the associated type I and type II error.
\end{remark}

\section{Length of confidence intervals and probability of non-intersection}
The first step of our analysis is to study some basic properties of our confidence intervals. In this paper we have two concurrent objectives: having tight confidence intervals (which we will also link to the power of the test) and having small probability of non-overlap under $H_0$. These two objectives conflict with one another because the smaller the confidence intervals, the larger the probability of non-overlapping. Hence, we will need to tune the sequence of $W_t$ to fit our framework. The first step towards this goal is to bound the length of the confidence intervals and to  control the distance between the two confidence intervals under $H_0$.

\subsection{Length of confidence intervals and calibration of the weights $W_t$}\label{sec:weights}
The following lemma controls the length of the confidence interval $C_n(\alpha; X,W)$ in the simple case in which the weights $\{W_t\}_{t}$ are all equal.
\begin{lemma}[Length of confidence interval -- constant weights]\label{lem:length}
Let $w>0$ and $n \in \N$, suppose that for all $1\le t\le n$, we have $W_t(X)=w$, and define $v = w^2/(1-w(b_P-a_P))$. Suppose that 
$$ nw^2 \ge   2v\left(\frac{nv}{2}\widehat{\sigma}_n^2 +\log(2/\alpha)\right),$$
where $\widehat{\sigma}_n^2 = \frac{1}{n}\sum_{i=1}^n(X_i - \frac{1}{n}\sum_{i=1}^nX_i)^2$ is the empirical variance.
Then the length $L(C_n(\alpha;X,W))$ of the confidence interval $C_n(\alpha;X,W)$ satisfies 
$$L(C_n(\alpha;X,W)) \le   2\frac{nw - \sqrt{n^2w^2 - 2nv\left(\frac{nv}{2}\widehat{\sigma}_n^2 +\log(2/\alpha)\right)}}{nv}.$$
\end{lemma}

% \subsection{Calibration of the predictable weights $W_t$}
Lemma~\ref{lem:length} offers some insight on how to choose the weights. Suppose that $n$ is large enough to have that the condition of Lemma~\ref{lem:length} is satisfied, as $n$ goes to infinity we have that
\begin{equation}\label{eq:length_hoeffding_1}
L(C_n(\alpha; X, W)) \lesssim 2\frac{\left(\frac{nw^2}{2(1-w(b_P-a_P))}\widehat{\sigma}_n^2 +\log(2/\alpha)\right)}{nw},
\end{equation}
where $\lesssim$ is taken here to mean ``for $n$ large enough''. Minimizing the right-hand-side of Equation~\eqref{eq:length_hoeffding_1}, we arrive at
% \begin{align*}
% \frac{1}{2(1-w(b_P-a_P))^2}\widehat{\sigma}_n^2  - \frac{1}{n w^2}\log(2/\alpha) = 0
% \end{align*}
% Hence, 
% \begin{align*}
% \frac{2(1-w(b_P-a_P))^2}{\widehat{\sigma}_n^2 } = \frac{n w^2}{\log(2/\alpha)}\\
% \frac{(1-w(b_P-a_P))}{\sqrt{\widehat{\sigma}_n^2} } = w\sqrt{\frac{n }{2\log(2/\alpha)}}\\
% (1-w(b_P-a_P)) = w\sqrt{\frac{n \widehat{\sigma}_n^2}{2\log(2/\alpha)}}\\
% 1 = w \left((b_P-a_P) + \sqrt{\frac{n \widehat{\sigma}_n^2}{2\log(2/\alpha)}} \right)
% \end{align*}
$$w = \frac{1}{(b_P-a_P) + \sqrt{\frac{n \widehat{\sigma}_n^2}{2\log(2/\alpha)}}}.$$
Based on this equation, we could define the following weights for fixed-time Bernstein's type confidence interval
\begin{equation}\label{eq:lambda_fixed_time}
W_t(X) =  \frac{1}{(b_P-a_P) + \sqrt{\frac{n \widehat{\sigma}_{t-1}^2}{2\log(2/\alpha)}}}.
\end{equation}
% Remark that as $n$ becomes big, these weights are equivalent to the weights chosen in \cite[Remark 3]{waudby2024estimating} (we do not simplify our weights to match theirs because we prefer to keep weights smaller than $1/(b_P-a_P)$).
However, we will see in Section~\ref{sec:power} that the power of the overlap test with this type of weights cannot be controlled in our setting in which the variance is unknown. Instead, we will use that $\widehat{\sigma}_n^2(x)$ is always smaller than $\frac{1}{4(b_P-a_P)^2}$ and define the following weights (see also Proposition 1 in \cite{waudby2024estimating}).
\begin{equation}\label{eq:weight_hoeffding_fixed_time}
w_t^f(X) =\frac{1}{b_P-a_P}\frac{1}{1+ \sqrt{\frac{n}{8c^2\log(2/\alpha)}}},
\end{equation}
where $c\ge1$ is a constant. 
\begin{remark}[Hoeffding-type weights]
If $c=1$ we call the weights in Equation~\eqref{eq:weight_hoeffding_fixed_time} Hoeffding-type weights as this amount to bounding $\widehat{\sigma}_n^2$ by $1/(4(b_P-a_P))$ in Equation~\eqref{eq:lambda_fixed_time}. The term ``Hoeffding-type'' weights may be misleading because $C_n(\alpha; X, w^f)$ improves upon Hoeffding inequality. Indeed, for $n$ large, the first order term for the length of the confidence interval with weights given by Equation~\eqref{eq:weight_hoeffding_fixed_time} is asymptotically 
$$C_n(\alpha; X, w)\le 2 \left(\left(\frac{\widehat{\sigma}_n^2}{b_P-a_P}\right)\sqrt{\frac{2\log(2/\alpha)}{n}} + (b_P-a_P)\sqrt{\frac{\log(2/\alpha)}{8n}}\right),$$
which is strictly better than Hoeffding's inequality by a factor at most $2$. 
\end{remark}

\begin{remark}[Tuning the constant $c$]
The constant $c$ is a tuning constant, in principle if we knew the standard deviation $\sigma$, we would set $c$ to $1/(2\sigma)$ and recover weights close to Equation~\eqref{eq:lambda_fixed_time}. However in general $\sigma$ is not known, then a safe value is $c=1$ to recover a Hoeffding-type confidence interval. If instead we think that the data will be well concentrated we can use a larger $c$. Whatever the value of $c$, $C_n(\alpha; X, w)$ remain a valid confidence interval, $c$ impacts on the theoretical error bounds of the test and the width of the confidence intervals.
\end{remark}

In the case of anytime confidence interval used in Algorithm~\ref{algo:anytime}, we follow the same reasoning as \cite[Section 3.1]{waudby2024estimating} replacing the $n$ from Equation~\eqref{eq:lambda_fixed_time} by a $t\log(1+t)$. We also introduce a period of deterministic weight before some time $t_0\ge1$.
\begin{equation}\label{eq:lambda_any_time}
W_t^{a}(X) =
\begin{cases}
\frac{1}{b_P-a_P}\left(\frac{1}{1+ \sqrt{\frac{t\log(t+1)}{8c^2\log(2/\alpha)}}}\right) & \text{ if }t \le t_0\\
\frac{1}{(b_P-a_P) + \sqrt{\frac{t\log(1+t)\widehat{\sigma}_{t-1}^2}{2\log(2/\alpha)}}} \wedge \frac{1}{b_P-a_P}\left(\frac{1}{1+ \sqrt{\frac{t_0\log(t+1)}{8c^2\log(2/\alpha)}}}\right)  & \text{ if }t > t_0
\end{cases}
\end{equation}
$W_t^{a}(X)$ is a predictable sequence tuned to get a confidence interval asymptotically of size $O(\sqrt{\log(t)/t})$ as explained in \cite[Section 3.1]{waudby2024estimating}. We will see when bounding the overlap probability (Section~\ref{sec:typeI}) the reason for introducing $t_0$ and $c$.

\subsection{Probability of non-intersection under $H_0$}\label{sec:proba_overlap}

% The main result of this paper is the following from which we derive the probability that the distance between the two confidence intervals is larger than some $\ell_{n,m}$. 
The following lemma controls the probability of two intervals being separated by more than a given value $\ell_{n,m}>0$.

\begin{lemma}\label{lem:elln}
Let $X_1,X_2,\dots$ be i.i.d. with law $P$ with support $[a_P, b_P]$ and $Y_1,Y_2,\dots$ be i.i.d. with law $Q$with support $ [a_Q, b_Q]$ such that $\E_P[X] = \E_Q[X] = \mu$. Let $n,m \in \mathbb{N}$, $\eta>0$ and $\eta'>0$ and let $\ell_{n,m} \ge 0$ be a real random variable such that
\begin{align*}
\ell_{n,m} \ge \frac{\log(1/\alpha)}{2}\max&\left(\frac{\eta'}{\sum_{t=1}^n \frac{W_t(X)}{1+W_t(X)(b_P-a_P)} }-\frac{\eta}{\sum_{t=1}^n \frac{W_t(X)}{1-W_t(X)(b_P-a_P)}} \right.\\
&\phantom{=}\left., \frac{\eta'}{\sum_{t=1}^m \frac{W_t(Y)}{1+W_t(Y)(b_Q-a_Q)} }-\frac{\eta}{\sum_{t=1}^m \frac{W_t(Y)}{1-W_t(Y)(b_Q-a_Q)}} \right).
\end{align*}
Then,
$$\P\left(\exists n,m\ge1 : \ C_n(\alpha;X,W)\le   C_m(\alpha;Y,W) - \ell_{n,m} \right)\le  \alpha^{\eta'+1} + \frac{1}{4}\alpha^{2-\eta},$$
and
$$\P\left(\exists n,m\ge1 : \ C_n(\alpha;X,W)\ge    C_m(\alpha;Y,W) + \ell_{n,m} \right)\le  \alpha^{\eta'+1} + \frac{1}{4}\alpha^{2-\eta}.$$
\end{lemma}
The proof is given in Section~\ref{sec:proof_lem_elln}.
Ideally, we want to bound the probabilities of non-overlapping, so $\ell_n=0$, which means taking $\eta$ close to $\eta'$. Then, because the weights $W_t(X)$ and $W_t(Y)$ go to $0$ when $t$ goes to infinity, this means we want to take $\eta$ and $\eta'$ both close to $1/2$. There is a small adjustment needed, and this adjustment depends on an upper bound of the sequence $\{W_t\}_{t\ge1}$ and this leads us to a trade-off: small $W_t$ are needed to get a manageable probability of error and (somewhat) large $W_t$ are needed to reduce the width of confidence intervals.

\section{Bounds on the error probabilities of the overlap tests}
In this section we bound the probabilities of error of the overlap test. In analogy with two-sample tests, we call type I errors the probabilities to choose either $H_1^-$ or $H_1^+$ when $H_0$ is true, type II error the probability to choose $H_0$ when either $H_1^-$ or $H_1^+$ is true and finally type III errors, the probabilities to choose $H_1^-$ when $H_1^+$ is true or to choose $H_1^+$ when $H_1^-$ was true. 

Remark that by summing the two type I errors we get the type I error of the bilateral test that would reject when the confidence intervals are disjoint, and the type II error of the bilateral test is bounded by the type II error of the three hypotheses test.
\subsection{Bound on type I error}\label{sec:typeI}
Using the results from Section~\ref{sec:proba_overlap}, we have the following bound on the type I errors.
\begin{theorem}[Type I errors]\label{th:typeIerror}
Suppose $\E_P[X] = \E_Q[Y]$. Let $t_0 \ge 0$ and suppose $W_t(X)$ and $W_t(Y)$ are deterministic for $t\le t_0$. Let
$$C_{t_0} =\max\left(\frac{\sum_{t=1}^{t_0} \frac{W_t(X)}{1-W_t(X)(b_P-a_P)}}{\sum_{t=1}^{t_0} \frac{W_t(X)}{1+W_t(X)(b_P-a_P)}}, \frac{\sum_{t=1}^{t_0} \frac{W_t(Y)}{1-W_t(Y)(b_Q-a_Q)}}{\sum_{t=1}^{t_0} \frac{W_t(Y)}{1+W_t(Y)(b_Q-a_Q)}}\right).$$ 
then,
$$\P(\exists n,m \ge t_0 \mid C_n(\alpha;X,W) > C_m(\alpha;Y,W))\le  \alpha\left(\frac{\alpha C_{t_0}}{4}\right)^{\frac{1}{1+C_{t_0}}}+ \frac{1}{4}\alpha^{2}\left(\frac{4}{\alpha C_{t_0}}\right)^{\frac{C_{t_0}}{1+C_{t_0}}},$$
and 
$$\P(\exists n,m \ge t_0 \mid C_n(\alpha;X,W) < C_m(\alpha;Y,W))\le  \alpha\left(\frac{\alpha C_{t_0}}{4}\right)^{\frac{1}{1+C_{t_0}}}+ \frac{1}{4}\alpha^{2}\left(\frac{4}{\alpha C_{t_0}}\right)^{\frac{C_{t_0}}{1+C_{t_0}}}.$$
\end{theorem}

Remark that Theorem~\ref{th:typeIerror} can be used to control the type I errors of both anytime and fixed-time test: by considering the test only until $t_0$, we revert to fixed-time test.
Remark that if $W_t(X)$ and $W_t(Y)$ are deterministic sequences that goes to $0$ as $t$ goes to infinity, we have that $C_{t_0}$ goes to $1$ and $\P(C_n(\alpha;X,W) >  C_m(\alpha;Y,W))$ goes to $\alpha^{3/2}$. On the other hand, when $C_{t_0}$ going to infinity, we get a probability of error of $\alpha$.

The proof is in Section~\ref{sec:proof_typeI_simple}. Remark that this theorem shows the importance of choosing $c$ not too large, although a constant $c$ that is too small would mean larger confidence intervals.

\subsection{Bound on type II error}\label{sec:power}
Bounding the type II error for the anytime test in Algorithm~\ref{algo:anytime} is a simple consequence of $C_n(\alpha:X,W)$ and $C_n(\alpha:X,W)$ being confidence intervals and this leads to the following theorem.
\begin{theorem}[Anytime bound on type II error]\label{th:typeII_anytime}
Let $\Delta>0$. Suppose $X_1,X_2,\dots$ are i.i.d. with law $P$ and $Y_1,Y_2,\dots$ are i.i.d. with law $Q$ both with bounded support and satisfy $\Delta\le|\E_P[X]-\E_Q[Y]|$. Then, the test in Algorithm~\ref{algo:anytime} with parameters $\Delta$ and $\alpha$ has the following type II error:
$$\P\left(\text{Accept }H_0\right)\le 2\alpha.$$
\end{theorem}
\begin{proof}
On the event ``Accept $H_0$'', we have $L(C_n(\alpha;X)\cup C_m(\alpha;Y))\le \Delta$ and $C_n(\alpha;X)\cap C_m(\alpha;Y) \neq \emptyset$. Hence, necessarily we have that either $\E_P[X] \notin C_n(\alpha;X)\cup C_m(\alpha;Y)$ or $\E_Q[Y] \notin C_n(\alpha;X)\cup C_m(\alpha;Y)$. Hence, either $\E_P[X] \notin C_n(\alpha;X)$ or $\E_Q[Y] \notin C_m(\alpha;Y)$, and this can happen with probability at most $2\alpha$.
\end{proof}

In the case of Fixed-time test, we need to bound the length of the confidence interval to get a bound on the type II error. From Lemma~\ref{lem:length}, we have
\begin{equation}\label{eq:length_hoeffding}
L(C_n(\alpha;X,w^f))\le L_H(P):=2\frac{nw - \sqrt{n^2w^2 - 2nv\left(\frac{nv}{8}+\log(2/\alpha)\right)}}{nv} .
\end{equation}
Remark that this bound is the reason for choosing deterministic weights for Fixed-time tests. Indeed, if we had chosen Bernstein-type weights, we could then bound $L(C_n(\alpha;X,w^f))$ with high probability results but to do this we would need to know the variances of $P$ and $Q$ which is not the case. Then, from Equation~\eqref{eq:length_hoeffding}, we  derive a bound on the type II error of our test.
\begin{theorem}[Fixed time bound on type II error]\label{th:typeII_hoeffding}
Suppose $X_1,\dots,X_n$ are i.i.d. with law $P$ and $Y_1,\dots,Y_n$ are i.i.d. with law $Q$, with 
$$\Delta:=\E_Q[Y]-\E_P[X]> L_H(P)+L_H(Q)$$
Then, 
\begin{align*}
\P&\left(C_n(\alpha;X, w^H)\cap C_m(\alpha:Y, w^H) \neq \emptyset\right)\\
&\le \alpha\left(e^{-(\Delta-L_H(P)-L_H(Q)) \sum_{t=1}^n \frac{w_t^f(P)}{1+w_t^f(P)(b_P-a_P)}}+ e^{-(\Delta-L_H(P)-L_H(Q)) \sum_{t=1}^n \frac{w_t^f(Q)}{1+w_t^f(Q)(b_Q-a_Q)}}\right).
\end{align*}
\end{theorem}
The proof is given in Section~\ref{sec:proof_th_typeII_hoeffding}. Remark that Theorem~\ref{th:typeII_hoeffding} is often inapplicable if $n$ and $m$ are not large enough so that $\E_Q[Y]-\E_P[X]\le  L_H(P)+L_H(Q)$. In practice this means we have no type II guarantee for the fixed-time test of overlap when $n$ or $m$ is too small. This is one of the main limitation of our method. 

\subsection{Bound on type III error}\label{sec:typeIII}
Finally, we bound the type III error. This type of error does not happen in a bilateral test and is often assumed too small to care in practice.
\begin{theorem}[Bound on type III errors]\label{th:typeIII}
Suppose $\E_P[X] > \E_Q[Y] + \Delta$ and for $t\le t_0$, suppose that $W_t(X)$ and $W_t(Y)$ are deterministic. Then 
$$\P\left(\exists n, m \ge t_0 \mid C_n(\alpha;X) < C_m(\alpha;Y)\right)\le \alpha^2 + \alpha\left(e^{-\Delta\sum_{t=1}^{t_0} \frac{W_t(Y)}{1+W_t(Y)(b_Q-a_Q)}}+e^{-\Delta\sum_{t=1}^{t_0} \frac{W_t(X)}{1+W_t(X)(b_P-a_P)}}\right).$$
And if $\E_P[X] < \E_Q[Y] - \Delta$, we have
$$\P\left(\exists n, m \ge t_0 \mid C_n(\alpha;X) > C_m(\alpha;Y)\right)\le \alpha^2 + \alpha\left(e^{-\Delta\sum_{t=1}^{t_0} \frac{W_t(Y)}{1+W_t(Y)(b_Q-a_Q)}}+e^{-\Delta\sum_{t=1}^{t_0} \frac{W_t(X)}{1+W_t(X)(b_P-a_P)}}\right).$$
\end{theorem}
The proof is given in Section~\ref{sec:proof_th_typeIII}. In general, the type III error should be small as it is of order $e^{-\Delta \Omega(\sqrt{t_0})}$ for $t_0$ big if $w_t$ decrease slower than $1/\sqrt{t}$.

\section{Examples of application}\label{sec:xp}
In this section we first apply our tests to a simulated dataset, and then we illustrate their usage in Machine Learning to compare the performances of Reinforcement Learning and Bandit algorithms. For each test, a constant $c$ must be tuned. The larger the constant $c$, the more we think that the distribution is in fact well concentrated. In the case of Fixed-time test, we should have $c\in [0,1]$ while in the case of Anytime test, $c$ plays the role of a scale parameter for the distribution and should be larger than $1$. If by chance we have a good guess on the standard deviation $\sigma$ of the distribution, we can use $c = 2/\sigma$ as parameter. Remark that whatever the choice of constant, the intervals are still valid confidence intervals and the theory holds true, the impact of $c$ is on the width of the resulting confidence interval and on the resulting error bounds. 
%To define our tests, we use the weights defined in Section~\ref{sec:weights} with $c_t = c/(1+t)^{1/4}$ for some parameter $c>0$, this heuristic choice is motivated by practical performances and by the need to have a sequence going slowly to $0$.

\subsection{Simulated dataset}
We consider various settings in which we draw two samples from each law and apply our tests. For each setting, we report the probabilities of errors in Figures~\ref{fig:results_anytime} and \ref{fig:results_fixed_time}. We consider the following settings.
\textbf{Ber equal}: $P$ and $Q$ are both $\mathrm{Benoulli}(1/2)$.
\textbf{Ber lower}: $P = \mathrm{Benoulli}(1/2)$ and $Q = \mathrm{Benoulli}(1/2 + 0.1)$.
\textbf{Unif vs Ber equal}: $P = \mathrm{Benoulli}(0.6)$ and $Q = \mathrm{Uniform}(0,1.2)$.
\textbf{Unif vs Ber lower}: $P = \mathrm{Benoulli}(0.6)$ and $Q = \mathrm{Uniform}(0,1.4)$. 
\textbf{Beta equal}: $P = \mathrm{Beta}(10,30)$ and $Q = \mathrm{Beta}(1,3)$. 
\textbf{Beta lower}: $P = \mathrm{Beta}(10,30)$ and $Q = \mathrm{Beta}(10,15)$. These settings are meant to represent both high-variance regime (Bernoulli distributions), low variance regime (Beta distributions), and various distributions shape and supports.

We report the results of the anytime test with parameters $\alpha=0.1$, $c=1$, $t_0=338$\footnote{We used the heuristic $t_0=\log(1/\alpha)/\mathrm{kl}(1/2,1/2+\Delta)$, using the Bernoulli kl as a heuristic of the complexity of estimating the mean up to $\Delta$ precision.} in Figure~\ref{fig:results_anytime}, and the Fixed-time test for parameters $\alpha=0.1$, $c=1$ and sample size $1000$ for each $P$ and $Q$ (hence $2000$ samples in total in comparison to the last row of Figure~\ref{fig:results_anytime}) in Figure~\ref{fig:results_fixed_time}. Both experiments were run $5000$ times to get the results presented. We included the theoretical bounds\footnote{The theoretical bounds could in principle be different between settings because of different supports but here and in the rest of this section, in practice, the bounds are numerically very close.} on error in the title of the figures.

\begin{figure}[h]
\footnotesize
\begin{center}
\begin{tabular}{|l||r|r|r|r|r|r|}
\toprule
 & Ber equal & Ber lower & Unif vs Ber equal & Unif vs Ber lower & Beta equal & Beta lower \\
\midrule
lower & 0.0014 & 0.9986 & 0.0010 & 0.9992 & 0.0000 & 1.0000 \\
equal & 0.9970 & 0.0014 & 0.9980 & 0.0008 & 1.0000 & 0.0000 \\
larger & 0.0016 & 0.0000 & 0.0010 & 0.0000 & 0.0000 & 0.0000 \\
\hline
n samples & 3089 (1084) & 2112 (1275) & 2204 (545) & 1744 (804) & 856 (95) & 694 (16) \\
\bottomrule
\end{tabular}
\caption{Probabilities of obtaining each decision and mean sample size at decision (with std in parentheses) for the Anytime overlap test. Theoretical bounds on errors are: $\mathrm{type\ I} \le 0.052$, $\mathrm{type\ II}\le 0.2$, $\mathrm{type\ III}\le 0.014$ }\label{fig:results_anytime}
\end{center}
\end{figure}

\begin{figure}[h]
\footnotesize

\begin{center}
\begin{tabular}{|l||r|r|r|r|r|r|}
\toprule
 & Ber equal & Ber lower & Unif vs Ber equal & Unif vs Ber lower & Beta equal & Beta lower \\
\midrule
lower & 0.0002 & 0.8512 & 0.0002 & 0.8770 & 0.0000 & 1.0000 \\
equal & 0.9990 & 0.1488 & 0.9998 & 0.1230 & 1.0000 & 0.0000 \\
larger & 0.0008 & 0.0000 & 0.0000 & 0.0000 & 0.0000 & 0.0000 \\
\bottomrule
\end{tabular}
\caption{Probabilities of obtaining each decision. Fixed time, with 1000 samples each. Theoretical bounds on errors are: $\mathrm{type\ I} \le 0.04$, $\mathrm{type\ II}\le 1$, $\mathrm{type\ III}\le 0.01$}\label{fig:results_fixed_time}
\end{center}
\end{figure}

We see on Figure~\ref{fig:results_anytime} and Figure~\ref{fig:results_fixed_time} that the theoretical bounds on error are typically very conservative and the obtained error is much lower than the bound, especially for the Fixed-time test for which the type II error is not informative due to sample sizes that are too low. Remark that in both cases, the low variance cases (Beta distributions) are easier and need less sample to get a small error, this shows that our test are adaptive to the specific shape of the distribution. 

\subsection{Usage in comparison of Sequential Learning Algorithms}
In this section we illustrate the use of our tests for comparison of learning algorithms and we represent the results in Figures~\ref{fig:ale} and \ref{fig:bandits}. 
The first experiment (Figure~\ref{fig:ale}) reuses data from \cite{agarwal2021deep} which compares the performances of deep reinforcement learning algorithms learning to play Atari games. The performance metric is the optimality gap which represents the difference between a base score of $1$, representing the performances of a Human expert, with the score of the agent. The optimality gap is bounded between $0$ and $1$, and it is highly nonparametric because different games give different score distributions which likely results in a multi-modal distribution. For a reader wishing to compare, say, REM and DreamerV2, Figure~\ref{fig:ale} gives right away the conclusion that REM performs worse than DreamerV2. There are 275 measures of optimality gap for each algorithm, except M-IQN which got 165 and DreamerV2 which got 605. For this experiment, we used the parameters $\alpha=0.1$, and $c=4$.

\begin{figure}[h]
  \begin{center}
\subfigure[Comparison of Deep RL algorithms on Atari games. Theoretical bounds on errors are: $\mathrm{type\ I} \le 0.073$ (resp. $0.079$ for tests with M-IQN), $\mathrm{type\ II}\le 1$, $\mathrm{type\ III}\le 0.01$.\label{fig:ale}]{
    \includegraphics[width=0.55\textwidth]{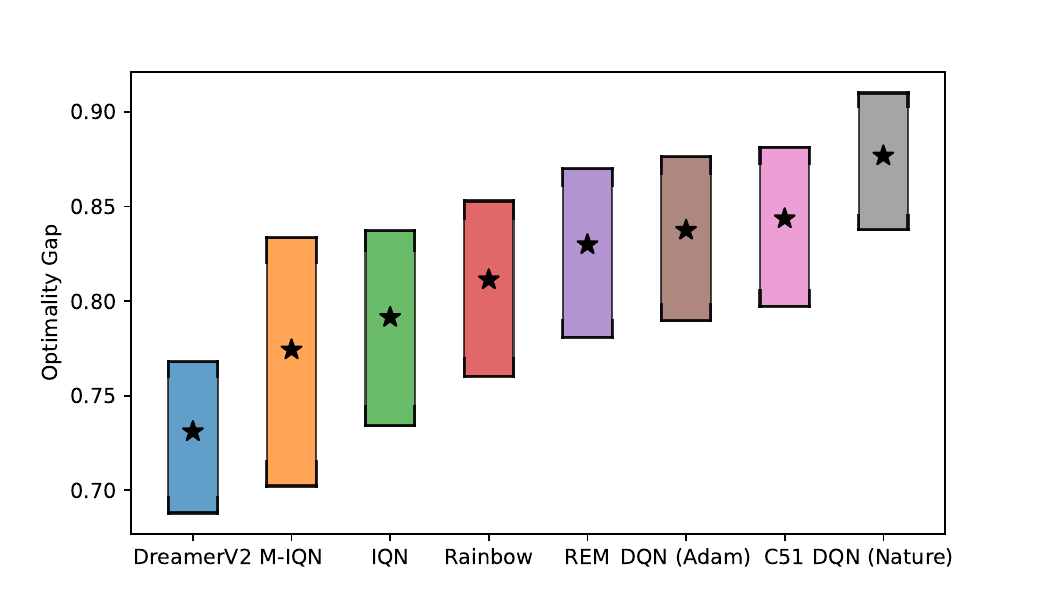}
    }
\subfigure[Comparison of bandit algorithms on a regret minimization problem with Bernoulli arms. Theoretical bounds on errors are: $\mathrm{type\ I} \le 0.088$, $\mathrm{type\ II}\le 0.2$, $\mathrm{type\ III}\le 0.4$.\label{fig:bandits}]{
\includegraphics[width=0.39\textwidth]{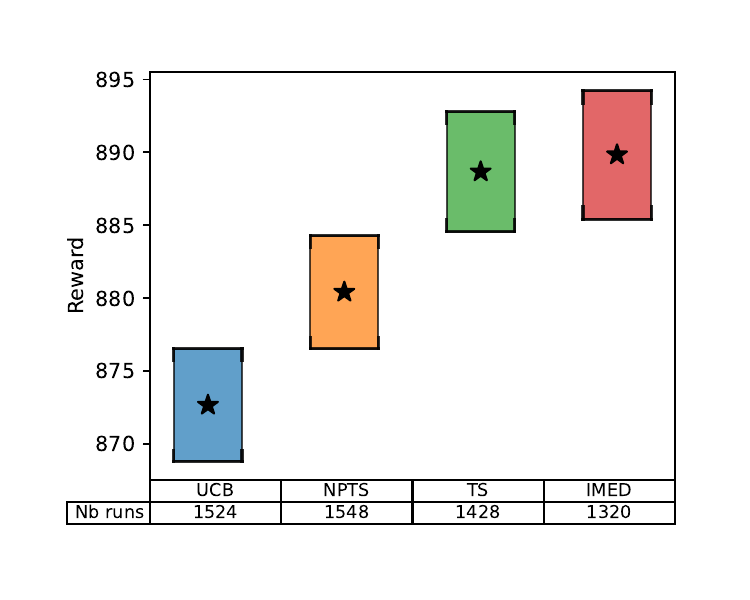}}
  \end{center}
  \caption{Confidence intervals for the comparison of Deep RL (Figure~\ref{fig:ale}) and Bandit (Figure~\ref{fig:bandits}) algorithms. The stars represent the empirical means and the boxes are the confidence intervals.}
\end{figure}

In Figure~\ref{fig:bandits} we represent the confidence intervals on the cumulative reward of several algorithms on a bandit problem with Bernoulli distribution of arms. The algorithms' the goal is to maximize the cumulative reward. The means of the arms are $0.6,0.85,0.9$ and the horizon is $1000$ (which can be considered a small sample regime for bandit problems). We use the anytime test defined in Algorithm~\ref{algo:anytime} to decide when to stop on each pairwise comparison, the confidence intervals and tests are comparing the cumulative reward of the algorithm all along the $1000$ steps of an algorithm's run. We do $16$ runs at a time getting the data in batch  in order to speed-up computation. At each step of the data collection process, we get more performance measurements for the algorithm whose confidence interval is the largest and for which the pairwise tests with other algorithms are not already decided. We used Algorithm~\ref{algo:anytime} with parameters $\alpha = 0.1$ $c=10$, $t_0 = 500$ and $\Delta=10$ (we expect the rewards to be well concentrated hence the value of $c$) to get the plot and the associated error bounds. Due to the nature of our tests, we stop gathering data (which here mean stop new training runs) when the confidence intervals are either disjoint or small enough to be able to accept that two means are equal (up to a precision $\Delta$). Remark that our test is not a multiple test, hence we have no control on the error when doing all the comparisons at the same time, we only have control on the error of one pairwise comparison.

\section{Conclusion}

In this article we provide a way to compare visually the means of the distributions of two samples using how much do their confidence intervals overlap, both in a sequential and non-sequential setting. We show that the tests we propose have controlled probabilities of error, and we show how our proposed method of comparison can be used in the context of comparing Reinforcement Learning and Bandit algorithms. In the statistics community it is well known that doing such a visual confidence interval is not optimal and that one should rather do a two-sample test to have a more powerful test, however we believe that our method will be of interest to practitioner and may be used informally in studies in which presentation and communication is important.

There are still limitations to our methods, in particular in the bound on type II error that seems very conservative for the Fixed-time test. Improving type II error guarantees is kept for future work.

% Acknowledgments---Will not appear in anonymized version
\acks{I want to thank the Scool team for the wonderful working environment that allowed me to write this nice article.
}

\bibliography{biblio}
\newpage
\appendix
\section{Proofs of Theorems}
% \crefalias{section}{appendix} % uncomment if you are using cleveref

\subsection{Proof of Theorem~\ref{th:typeIerror}}\label{sec:proof_typeI_simple}
From Lemma~\ref{lem:elln}, we are searching for the minimum of $e^{1+\eta'}+\frac{1}{4}e^{2-\eta}$ under the constraint that
$$\frac{\eta'}{\sum_{t=1}^n \frac{W_t(X)}{1+W_t(X)(b_P-a_P)} }-\frac{\eta}{\sum_{t=1}^n \frac{W_t(X)}{1-W_t(X)(b_P-a_P)}}\le 0$$
and 
$$\frac{\eta'}{\sum_{t=1}^m \frac{W_t(Y)}{1+W_t(Y)(b_Q-a_Q)}}-\frac{\eta}{\sum_{t=1}^m \frac{W_t(Y)}{1-W_t(Y)(b_Q-a_Q)}}\le 0.$$
We take 
$$ \eta \ge  \eta'\max\left(\frac{\sum_{t=1}^n \frac{W_t(X)}{1-W_t(X)(b_P-a_P)}}{\sum_{t=1}^n \frac{W_t(X)}{1+W_t(X)(b_P-a_P)}},\frac{\sum_{t=1}^m \frac{W_t(Y)}{1-W_t(Y)(b_Q-a_Q)}}{\sum_{t=1}^m \frac{W_t(Y)}{1+W_t(Y)(b_Q-a_Q)}}\right).$$
Now, using the inequality $\sum \alpha_i/\sum \beta_i \le \max_i \alpha_i/\beta_i$ for $\{\alpha_i\}_i$ and $\{\beta_i\}_i$ positives (see Lemma 12.5 in \cite{orabona2019modern}), we get
\begin{align*}
\frac{\sum_{t=1}^n \frac{W_t(X)}{1-W_t(X)(b_P-a_P)}}{\sum_{t=1}^n \frac{W_t(X)}{1+W_t(X)(b_P-a_P)}}&\le \max\left(\frac{\sum_{t=1}^{t_0} \frac{W_t(X)}{1-W_t(X)(b_P-a_P)}}{\sum_{t=1}^{t_0} \frac{W_t(X)}{1+W_t(X)(b_P-a_P)}}, \frac{\sum_{t=t_0}^{n} \frac{W_t(X)}{1-W_t(X)(b_P-a_P)}}{\sum_{t=t_0}^{n} \frac{W_t(X)}{1+W_t(X)(b_P-a_P)}}\right)\\
&\le \max\left(\frac{\sum_{t=1}^{t_0} \frac{W_t(X)}{1-W_t(X)(b_P-a_P)}}{\sum_{t=1}^{t_0} \frac{W_t(X)}{1+W_t(X)(b_P-a_P)}},\max_{t_0\le t\le n} \frac{\frac{W_t(X)}{1-W_t(X)(b_P-a_P)}}{\frac{W_t(X)}{1+W_t(X)(b_P-a_P)}}\right)
\end{align*}
Then, having that $W_t(X)$ is non-increasing, the first term of the maximum is the largest and for any $t \le t_0$, 
\begin{align*}
\frac{\sum_{t=1}^n \frac{W_t(X)}{1-W_t(X)(b_P-a_P)}}{\sum_{t=1}^n \frac{W_t(X)}{1+W_t(X)(b_P-a_P)}}&\le\frac{\sum_{t=1}^{t_0} \frac{W_t(X)}{1-W_t(X)(b_P-a_P)}}{\sum_{t=1}^{t_0} \frac{W_t(X)}{1+W_t(X)(b_P-a_P)}}.
\end{align*}
We have the same for $W_t(Y)$. Then, we denote 
$$C_{t_0} =\max\left(\frac{\sum_{t=1}^{t_0} \frac{W_t(X)}{1-W_t(X)(b_P-a_P)}}{\sum_{t=1}^{t_0} \frac{W_t(X)}{1+W_t(X)(b_P-a_P)}}, \frac{\sum_{t=1}^{t_0} \frac{W_t(Y)}{1-W_t(Y)(b_Q-a_Q)}}{\sum_{t=1}^{t_0} \frac{W_t(Y)}{1+W_t(Y)(b_Q-a_Q)}}\right).$$ 
We want to minimize for $\eta'$ the function
$$J(\eta') = \alpha^{1+\eta'}+ \frac{1}{4}\alpha^{2 - C_{t_0} \eta'} $$
Taking the derivative of $J$, we get that the optimal $\eta'$ satisfies the equation
$$\log(\alpha)\alpha^{1+\eta'} - \frac{C_{t_0}\log(\alpha)\alpha^{2 - C_{t_0}\eta'}}{4}=0$$
Hence, $\alpha^{\eta'(1+C_{t_0}) - 1} = \frac{C_{t_0}}{4}$, which gives us 
$$\alpha^{\eta'} = \left(\frac{\alpha C_{t_0}}{4}\right)^{\frac{1}{1+C_{t_0}}}.$$
Injecting this into the error probability gives the result. 

\subsection{Proof of Theorem~\ref{th:typeII_hoeffding}}\label{sec:proof_th_typeII_hoeffding}
From Equation~\eqref{eq:length_hoeffding}, we have 

$$L(C_n(\alpha;X,w^f))+ L(C_m(\alpha;Y,w^f))\le L_H(P)+L_H(Q) \le \Delta$$
Hence, if the two interval intersect, we know that we cannot have at the same time $\E_P[X] \in C_n(\alpha;X,w^f)$ and $ \E_Q[Y] \in C_m(\alpha;Y,w^f)$. In fact we have even more, let $[A(X),B(X)] = C_n(\alpha;X,w^f)$ and $[A(Y),B(Y)] = C_n(\alpha;Y,w^f)$ such that $E_n(A(X);X,W) = E_n(B(X);X,W) = 1/\alpha$, and similarly $[A(Y),B(Y)] = C_m(\alpha;Y,w^f)$. We have
\begin{itemize}
\item If $\E_P[X] \in C_n(\alpha;X,w^f)$, having $\E_Q[Y]-\E_P[X] = \Delta \ge L(C_n(\alpha,X,w^f))+ L(C_m(\alpha,Y,w^f))$, we get $\E_Q[Y] \ge \Delta + \left(B(Y)- (L(C_n(\alpha,X,w^f))+ L(C_m(\alpha,Y,w^f)))\right)$. Hence, using Lemma~\ref{lem:bound_deriv},
\begin{align*}
\log(E_m(\E_Q[Y];Y)) &\ge \log(E_m\left(B(Y) + \Delta - \left(L(C_n(\alpha,X,w^f))+ L(C_m(\alpha,Y,w^f))\right);Y\right)) \\
&\ge \log(E_m\left(B(Y) + \Delta - \left(L_H(P)+L_H(Q)\right);Y\right))\\
&\ge \log(E_m\left(B(Y) ;Y\right)) + (\Delta - \left(L_H(P)+L_H(Q)\right)) \inf_{z} \left|\frac{\d }{\d z} \log(E_m\left(z;Y\right)) \right|\\
&\ge \log\left(\frac{1}{\alpha}\right) + (\Delta-\left(L_H(P)+L_H(Q)\right)) \sum_{t=1}^n \frac{w_t^f(Q)}{1+w_t^f(Q)(b_Q-a_Q)}.
\end{align*}
\item If $\E_Q[Y] \in C_n(\alpha;X,w^f)$, similarly, we have 
\begin{align*}
\log(E_n(\E_Q[X];X)) &\ge \log(E_n\left(A(X) - (\Delta - \left(L_H(P)+L_H(Q)\right));X\right)) \\
&\ge \log\left(\frac{1}{\alpha}\right) + (\Delta-\left(L_H(P)+L_H(Q)\right)) \sum_{t=1}^n \frac{w_t^f(P)}{1+w_t^f(P)(b_P-a_P)}.
\end{align*}
\end{itemize}
Hence, by Markov inequality's and Union bound, we have
\begin{align*}
\P&\left(C_n(\alpha;X,w^f) \cap C_m(\alpha;Y,w^f) \neq \emptyset\right)\\
&\le  \P\left(\log(E_m(\E_Q[Y];Y))  \ge \log\left(\frac{1}{\alpha}\right) + (\Delta-L_H(P)-L_H(Q)) \sum_{t=1}^n \frac{w_t^f(Q)}{1+w_t^f(Q)(b_Q-a_Q)}\right) \\
&\phantom{=}+ \P\left(\log(E_n(\E_Q[X];X))  \ge \log\left(\frac{1}{\alpha}\right) +(\Delta-L_H(P)-L_H(Q)) \sum_{t=1}^n \frac{w_t^f(P)}{1+w_t^f(P)(b_P-a_P)}\right)\\
&\le \alpha\left(e^{-(\Delta-L_H(P)-L_H(Q)) \sum_{t=1}^n \frac{w_t^f(P)}{1+w_t^f(P)(b_P-a_P)}}+ e^{-(\Delta-L_H(P)-L_H(Q)) \sum_{t=1}^n \frac{w_t^f(Q)}{1+w_t^f(Q)(b_Q-a_Q)}}\right).
\end{align*}
\subsection{Proof of Theorem~\ref{th:typeIII}}\label{sec:proof_th_typeIII}
We decompose the event $C_n(\alpha;X) < C_m(\alpha;Y)$ into three events (Events 1 and 2 are represented in Figure~\ref{fig:typeIII}):
\begin{description}
\item[Event 1] $\E_Q[Y] \in C_m(\alpha;Y)$, and as a consequence $\E_P[X] > C_n(\alpha;X)+\Delta$,
\item[Event 2] $\E_P[X] \in C_n(\alpha;X)$ , and as a consequence $\E_Q[Y] < C_m(\alpha;Y)-\Delta$,
\item[Event 3] $\E_P[X] \notin C_n(\alpha;X)$ and $\E_Q[Y] \notin C_m(\alpha;Y)$.
\end{description}

Event 3 has a probability smaller than $\alpha^2$ because the two samples $X_1,\dots,X_n$ and $Y_1,\dots,Y_m$ are independent, and both $C_n(\alpha;X)$ and $C_m(\alpha;Y)$ are confidence intervals of level $1-\alpha$

\begin{figure}[h]
  \begin{center}
    % -*- TeX-master: "article.tex" -*-
% \subfigure[Event 1]{
\minipage{0.45\textwidth}

\begin{tikzpicture}[font=\small, scale=0.8] 
% \draw[-Stealth] (-.5,0) -- (9.5,0) ;
% \foreach \x in {0,1,...,9}
%   \draw[shift=(right:\x)] (0pt,3pt) -- (0pt,-3pt) node[below] {$\x$};

\path (5,0) to[bracket-bracket=80] (8,0);
\fill[opacity = 0.2, blue] (5,-0.15) -- (8, -0.15) -- (8, 0.15) -- (5,0.15) -- cycle;
\node at (6.5,-.5) {$C_\alpha(Y_1^n)$};

\path (1,0)  to[bracket-bracket=80]  (4,0) ;
\fill[opacity = 0.2, blue] (1,-0.15) -- (4, -0.15) -- (4, 0.15) -- (1,0.15) -- cycle;
\node at (2.5,-.5) {$C_\alpha(X_1^n)$};
% \draw (0,1) -- (2,-1);
% \draw (2,-1) -- (4,1);

% \draw[-, ultra thick, blue] (1,0) -- (3.5,0);

\draw[-Stealth] (0.5,0) coordinate (1,0) -- (9,0) coordinate (end);

% \draw[Stealth-Stealth] (4.5,0.5) --  node [midway,above]{$\ell_{n,m}$}  (5.5,0.5);
% \draw[densely dotted] (4.5,0) -- (4.5,0.5);
% \draw[densely dotted] (5.5,0) -- (5.5,0.5);

% \draw[shift=(right:4.5)] (0pt,3pt) -- (0pt,-3pt) node[below] {$\widehat{F}$};
\draw[shift=(right:6)] (0pt,3pt) -- (0pt,-3pt) node[above, yshift=0.4em] {$\E_Q[Y]$};
\draw[shift=(right:7.8)] (0pt,3pt) -- (0pt,-3pt) node[above, yshift=0.4em] {$\E_P[X]$};

\draw[Stealth-Stealth] (6,1) --  node [midway,above]{$\Delta$}  (7.8,1);
\draw[densely dotted] (6,0) -- (6,1);
\draw[densely dotted] (7.8,0) -- (7.8,1);

% \draw[decorate,decoration={brace,raise=0.75cm}] (6,0) -- (4.5,0) ;
% \node at (5.25,-1) {$L$};

% \foreach \x in {0,1,...,9}
%   \draw[shift=(right:\x)] (0pt,3pt) -- (0pt,-3pt) node[below] {$\x$};

% \path (start) to[-bracket  =-80] (2,0);
% \path (7,0)   to[bracket-=-80] (end);
\end{tikzpicture}
% }
\captionof{figure}{Event 1}
\endminipage
\hspace{0.5em}
% \subfigure[Event 2]{
\minipage{0.45\textwidth}

\begin{tikzpicture}[font=\small, scale=0.8] 
% \draw[-Stealth] (-.5,0) -- (9.5,0) ;
% \foreach \x in {0,1,...,9}
%   \draw[shift=(right:\x)] (0pt,3pt) -- (0pt,-3pt) node[below] {$\x$};

\path (5,0) to[bracket-bracket=80] (8,0);
\fill[opacity = 0.2, blue] (5,-0.15) -- (8, -0.15) -- (8, 0.15) -- (5,0.15) -- cycle;
\node at (6.5,-.5) {$C_\alpha(Y_1^n)$};

\path (1,0)  to[bracket-bracket=80]  (4,0) ;
\fill[opacity = 0.2, blue] (1,-0.15) -- (4, -0.15) -- (4, 0.15) -- (1,0.15) -- cycle;
\node at (2.5,-.5) {$C_\alpha(X_1^n)$};
% \draw (0,1) -- (2,-1);
% \draw (2,-1) -- (4,1);

% \draw[-, ultra thick, blue] (1,0) -- (3.5,0);

\draw[-Stealth] (0.5,0) coordinate (1,0) -- (9,0) coordinate (end);

% \draw[Stealth-Stealth] (4.5,0.5) --  node [midway,above]{$\ell_{n,m}$}  (5.5,0.5);
% \draw[densely dotted] (4.5,0) -- (4.5,0.5);
% \draw[densely dotted] (5.5,0) -- (5.5,0.5);

% \draw[shift=(right:4.5)] (0pt,3pt) -- (0pt,-3pt) node[below] {$\widehat{F}$};
\draw[shift=(right:1.2)] (0pt,3pt) -- (0pt,-3pt) node[above, yshift=0.4em] {$\E_Q[Y]$};
\draw[shift=(right:3)] (0pt,3pt) -- (0pt,-3pt) node[above, yshift=0.4em] {$\E_P[X]$};

\draw[Stealth-Stealth] (1.2,1) --  node [midway,above]{$\Delta$}  (3,1);
\draw[densely dotted] (1.2,0) -- (1.2,1);
\draw[densely dotted] (3,0) -- (3,1);

% \draw[decorate,decoration={brace,raise=0.75cm}] (6,0) -- (4.5,0) ;
% \node at (5.25,-1) {$L$};

% \foreach \x in {0,1,...,9}
%   \draw[shift=(right:\x)] (0pt,3pt) -- (0pt,-3pt) node[below] {$\x$};

% \path (start) to[-bracket  =-80] (2,0);
% \path (7,0)   to[bracket-=-80] (end);
\end{tikzpicture}
\captionof{figure}{Event 2}
\endminipage
% }
\caption{Illustration of Event 1 and Event 2.\label{fig:typeIII}}
  \end{center}
\end{figure}
On the other hand, in Event 1, using Lemma~\ref{lem:bound_deriv}, we have 
\begin{align*}
\log(E_n(\E_P[X];X)) &\ge \log(E_n(\max C_n(\alpha;X)+\Delta;X)) \\
&\ge \log(1/\alpha) + \Delta\inf_{z} \left|\frac{\d }{\d z} \log(E_n(z;X))\right|\\
&\ge \log(1/\alpha) + \Delta \sum_{t=1}^n \frac{W_t(X)}{1+W_t(X)(b_P-a_P)}\\
&\ge \log(1/\alpha) + \Delta \sum_{t=1}^{t_0} \frac{W_t(X)}{1+W_t(X)(b_P-a_P)}.
\end{align*}
Similarly, on Event 2 
\begin{align*}
\log(E_m(\E_Q[Y];Y)) &\ge \log(1/\alpha) + \Delta \sum_{t=1}^{t_0} \frac{W_t(Y)}{1+W_t(Y)(b_Q-a_Q)},
\end{align*}
and by Ville's inequality and union bound, we have that the probability of either of these two inequalities ever happening for some $n,m\ge t_0$ is smaller than 
$$\alpha\left(e^{-\Delta\sum_{t=1}^{t_0} \frac{W_t(Y)}{1+W_t(Y)(b_Q-a_Q)}}+e^{-\Delta\sum_{t=1}^{t_0} \frac{W_t(X)}{1+W_t(X)(b_P-a_P)}}\right).$$

\section{Technical tools and proofs of Lemmas}
\subsection{Basic results on $E_n$}
We present in this Section two lemmas with basic properties of the function $E_n$.
\begin{lemma}\label{lem:convex}
The function $z \mapsto \log(E_n(z; X))$ is convex over $[a_P,b_P]$.
\end{lemma}
\begin{proof}
Let us denote $\widehat{\mu} \in \arg\min_{z \in \R} E_n(z; X)$ which also corresponds to the value of $z$ such that the two terms in the maximum of the difinition of $E_n$ (Equation~\eqref{eq:def_En}) are equal.

The function $z \mapsto \log(E_n(z;X))$ is differentiable everywhere except in $z = \widehat{\mu}$ and for $z >  \widehat{\mu}$, its derivative is  
\begin{align*}
\frac{\d }{\d z} \log(E_n(z;X))&= \sum_{t=1}^n \frac{-W_t(X)}{1+W_t(X)(X_t-z)}
\end{align*}
and if $z < \widehat{\mu}$,
\begin{align*}
\frac{\d }{\d z} \log(E_n(z;X))&=\sum_{t=1}^n \frac{W_t(X)}{1-W_t(X) (X_t-z)}.
\end{align*}
and then, we can compute the second derivative on each of the two portion of $[0,1]$ and get for $z \ge \widehat{\mu}$, 
\begin{align*}
\frac{\d^2 }{\d z^2} \log(E_n(z;X))&=\sum_{t=1}^n \frac{W_t(X)^2}{(1+W_t(X) (X_t-z))^2}>0
\end{align*}
and for $z < \widehat{\mu}$, similarly we get
\begin{align*}
\frac{\d^2 }{\d z^2} \log(E_n(z;X))&= \sum_{t=1}^n \frac{W_t(X)^2}{(1-W_t(X) (X_t-z))^2} >0 
\end{align*}
Hence, $z \mapsto \log(E_n(z;X))$ is convex. % Or maybe simpler: maximum of two convex functions is convex.
\end{proof}

\begin{lemma}\label{lem:bound_deriv}
Let $X_1,\dots,X_n$ be i.i.d. from a law $P$ with support in $[a_P,b_P]$ for some $a_P,b_P\in \R$, let $\widehat{\mu} \in \arg\min_{z} \log(E_n(z;X))$, for any $z \neq \widehat{\mu}$, $z \mapsto \log(E_n(z;X))$ is derivable with derivative bounded as follows
$$ \sum_{t=1}^n \frac{W_t(X)}{1+ W_t(X)(b_P-a_P)} \le \left|\frac{\d }{\d z} \log(E_n(z;X))\right| \le  \sum_{t=1}^n \frac{W_t(X)}{1- W_t(X)(b_P-a_P)}.$$
\end{lemma}
\begin{proof}
We have for any $z > \widehat{\mu}$
\begin{align*}
\frac{\d }{\d z} \log(E_n(z;X))&=   \sum_{t=1}^n \frac{W_t}{1-W_t (X_t-z)} 
\end{align*}
and using that $X_t - z \in [a_P,b_P]$, we get 
$$\sum_{t=1}^n \frac{W_t(X)}{1+ W_t(X)(b_P-a_P)} \le \frac{\d }{\d z} \log(E_n(z;X)) \le \sum_{t=1}^n \frac{W_t(X)}{1- W_t(X)(b_P-a_P)}$$

Similarly, for $z < \widehat{\mu}$,
\begin{align*}
\frac{\d }{\d z} \log(E_n(z;X))&=   \sum_{t=1}^n \frac{-W_t}{1+W_t (X_t-z)} 
\end{align*}
and then again using that $X_t$ is bounded, we get the result,
$$-\sum_{t=1}^n \frac{W_t(X)}{1- W_t(X)(b_P-a_P)} \le \frac{\d }{\d z} \log(E_n(z;X)) \le -\sum_{t=1}^n \frac{W_t(X)}{1+ W_t(X)(b_P-a_P)}.$$
\end{proof}

\subsection{Proof of Lemma~\ref{lem:length}}
The following inequality on the logarithm will be used to bound the length of the confidence interval.
\begin{lemma}[Inequality on the logarithm]\label{lem:log_ineq}
For any $x$ such that $|x|\le c\le 1$, we have $\log(1+x)\ge \frac{x(2-c)}{2-c+x}$. 
\end{lemma}
\begin{proof}
Let $f(x) = \log(1+x) - \frac{x(2-c)}{2-c+x}$. We have for any $x$ such that $|x|\le c$,
\begin{align*}
 f'(x) &= \frac{1}{1+x}-\frac{(2-c)(2-c+x) - x(2-c)}{(2-c+x)^2} \\
&= \frac{1}{1+x}-\frac{(2-c)^2}{(2-c+x)^2} \\
&= \frac{(2-c+x)^2 - (1+x)(2-c)^2}{(1+x)(2-c+x)^2}\\
&=  \frac{x\left(4-2c+ x- (4-4c+c^2)\right)}{(1+x)(2-c+x)^2}= \frac{x\left(2c+ x- c^2\right)}{(1+x)(2-c+x)^2}
\end{align*}
and having $|x|\le c\le 1$, we get that $f'(x)$ is non-positive for $-c\le x\le 0$ and non-negative for $0\le x\le c$. Then, having $f(0) = 0$ we get that $f(x)\ge 0$.
\end{proof}
 Hence, having that $w(X_i-z)\le w(b_P-a_P)$, from Lemma~\ref{lem:log_ineq},
\begin{align*}
\log(2E_n(z;X)) &\ge \sum_{t=1}^n \log\left(1 + w(X_t-z)\right) \\
&\ge \sum_{t=1}^n \frac{w(X_t-z)(2-w(b_P-a_P))}{2-w(b_P-a_P)+ w(X_t-z)}\\
&= w\sum_{t=1}^n (X_t-z) + \sum_{t=1}^nw(X_t-z)\left( \frac{2-w(b_P-a_P)}{2-w(b_P-a_P)+w(X_t-z)} - 1\right)\\
&= w\sum_{t=1}^n (X_t-z) - \sum_{t=1}^n\frac{w^2(X_t-z)^2}{2-w(b_P-a_P)+w(X_t-z)}\\
&\ge w\sum_{t=1}^n (X_t-z) - \sum_{t=1}^n\frac{w^2(X_t-z)^2}{2(1-w(b_P-a_P))}
\end{align*}
Where we used $X_t \in [a_P,b_P]$ for the last inequality. Let $v = \frac{w^2}{1-w(b_P-a_P)}$, we are then searching for $z$ such that 
\begin{equation}\label{eq:poly2_power} 
nw(\overline{X}_n-z)- \frac{v}{2} \sum_{t=1}^n(X_t-z)^2\ge \log(2/\alpha),
\end{equation}
And, by bias-variance decomposition, this reduces to
\begin{equation}\label{eq:poly2_power} 
nw(\overline{X}_n-z) - \frac{nv}{2} \left(\widehat{\sigma}_n^2 + \left(z - \overline{X}_n\right)^2\right)\ge \log(2/\alpha),
\end{equation}
This is a second order polynomial in $(\overline{X}_n-z)$. It has solutions if 
$$(nw)^2 \ge  2nv\left(\frac{nv}{2}\widehat{\sigma}_n^2 +\log(2/\alpha)\right) $$
which reduces to the condition on $n$
$$ nw^2 \ge  2v\left(\frac{nv}{2}\widehat{\sigma}_n^2 +\log(2/\alpha)\right) $$
and then, we solve the polynomial to have the condition
\begin{align*}
\overline{X}_n- z & \ge  \frac{nw - \sqrt{n^2w^2 - 2nv\left(\frac{nv}{2}\widehat{\sigma}_n^2 +\log(2/\alpha)\right)}}{nv}.
\end{align*}
Similarly, we have
\begin{align*}
\log(2E_n(z;X)) &\ge \sum_{t=1}^n \log\left(1 - w(X_t-z)\right) \\
&\ge -\sum_{t=1}^n \frac{w(X_t-z)(2-w(b_P-a_P))}{2-w(b_P-a_P)- w(X_t-z)}\\
&= -w\sum_{t=1}^n (X_t-z) - \sum_{t=1}^nw(X_t-z)\left( \frac{2-c}{2-w(b_P-a_P)-w(X_t-z)} + 1\right)\\
&\ge -w\sum_{t=1}^n W_t(X_t-z) - \frac{v}{2}\sum_{t=1}^n(X_t-z)^2
\end{align*}
and solving the second order polynomial, we arrive at the condition
\begin{align*}
z-\overline{X}_n & \ge  \frac{nw - \sqrt{n^2w^2 - 2nv\left(\frac{nv}{2}\widehat{\sigma}_n^2 +\log(2/\alpha)\right)}}{nv}.
\end{align*}
As a consequence, we have that the length of the interval on which $\log(2E_n(z;X)) \ge \log(2/\alpha)$ is bounded by 
$$2\frac{nw - \sqrt{n^2w^2 - 2nv\left(\frac{nv}{2}\widehat{\sigma}_n^2 +\log(2/\alpha)\right)}}{nv}$$
\subsection{Proof of Lemma~\ref{lem:elln}}\label{sec:proof_lem_elln}
Having decided $H_1^-$, we have $C_n(\alpha;X,W) \le C_m(\alpha;Y,W)$. Define $\widehat{F}$ the midpoint between the upper bound $C_n(\alpha;X,W)$ and the lower bound of $C_m(\alpha;Y,W)$.
Define the event 
$$E_L : \mu \in[\widehat{F}-L_X, \widehat{F}+L_Y]$$
with $L_X,L_Y$ that we will fix later on. We study separately the case $E_L$ true and $E_L$ false.
\paragraph{Case 1:} if $E_L$ is true, 
\begin{figure}[h]
\begin{center}
% -*- TeX-master: "article.tex" -*-
% \subfigure[Event 1]{
% \subcaptionbox{Event 1}[0.45\textwidth]{

\minipage{0.45\textwidth}
\begin{tikzpicture}[font=\small, scale=0.8] 
% \draw[-Stealth] (-.5,0) -- (9.5,0) ;
% \foreach \x in {0,1,...,9}
%   \draw[shift=(right:\x)] (0pt,3pt) -- (0pt,-3pt) node[below] {$\x$};

\path (5.5,0) to[bracket-bracket=80] (8,0);
\fill[opacity = 0.2, blue] (5.5,-0.15) -- (8, -0.15) -- (8, 0.15) -- (5.5,0.15) -- cycle;
\node at (6.75,-.5) {$C_\alpha(Y_1^n)$};

\path (1,0)  to[bracket-bracket=80]  (3.5,0) ;
\fill[opacity = 0.2, blue] (1,-0.15) -- (3.5, -0.15) -- (3.5, 0.15) -- (1,0.15) -- cycle;
\node at (2.25,-.5) {$C_\alpha(X_1^n)$};
% \draw (0,1) -- (2,-1);
% \draw (2,-1) -- (4,1);

% \draw[-, ultra thick, blue] (1,0) -- (3.5,0);

\draw[-Stealth] (0.5,0) coordinate (1,0) -- (9,0) coordinate (end);

\draw[Stealth-Stealth] (4.5,1.2) --  node [midway,above]{$L_Y$}  (6,1.2);
\draw[densely dotted] (4.5,0) -- (4.5,1.2);
\draw[densely dotted] (6,0) -- (6,1.2);

\draw[Stealth-Stealth] (4.5,0.5) --  node [midway,above]{$\ell_{n,m}$}  (5.5,0.5);
\draw[densely dotted] (4.5,0) -- (4.5,0.5);
\draw[densely dotted] (5.5,0) -- (5.5,0.5);

\draw[shift=(right:4.5)] (0pt,3pt) -- (0pt,-3pt) node[below] {$\widehat{F}$};
\draw[shift=(right:5.76)] (0pt,3pt) -- (0pt,-3pt) node[below] {$\mu$};

% \draw[decorate,decoration={brace,raise=0.75cm}] (6,0) -- (4.5,0) ;
% \node at (5.25,-1) {$L$};

% \foreach \x in {0,1,...,9}
%   \draw[shift=(right:\x)] (0pt,3pt) -- (0pt,-3pt) node[below] {$\x$};

% \path (start) to[-bracket  =-80] (2,0);
% \path (7,0)   to[bracket-=-80] (end);

\end{tikzpicture}
% \subcaption{Event 1}
% }
\captionof{figure}{$E_L$ true: Event 1}
\endminipage
\hspace{0.5em}
\minipage{0.45\textwidth}
 
% \subfigure[Event 2]{
% \subcaptionbox{Event 2}[0.45\textwidth]{
\begin{tikzpicture}[font=\small, scale=0.8] 
% \draw[-Stealth] (-.5,0) -- (9.5,0) ;
% \foreach \x in {0,1,...,9}
%   \draw[shift=(right:\x)] (0pt,3pt) -- (0pt,-3pt) node[below] {$\x$};

\path (5.5,0) to[bracket-bracket=80] (8,0);
\fill[opacity = 0.2, blue] (5.5,-0.15) -- (8, -0.15) -- (8, 0.15) -- (5.5,0.15) -- cycle;
\node at (6.75,-.5) {$C_\alpha(Y_1^n)$};

\path (1,0)  to[bracket-bracket=80]  (3.5,0) ;
\fill[opacity = 0.2, blue] (1,-0.15) -- (3.5, -0.15) -- (3.5, 0.15) -- (1,0.15) -- cycle;
\node at (2.25,-.5) {$C_\alpha(X_1^n)$};
% \draw (0,1) -- (2,-1);
% \draw (2,-1) -- (4,1);

% \draw[-, ultra thick, blue] (1,0) -- (3.5,0);

\draw[-Stealth] (0.5,0) coordinate (1,0) -- (9,0) coordinate (end);

\draw[Stealth-Stealth] (4.5,1.2) --  node [midway,above]{$L_X$}  (3,1.2);
\draw[densely dotted] (4.5,0) -- (4.5,1.2);
\draw[densely dotted] (3,0) -- (3,1.2);

\draw[Stealth-Stealth] (4.5,0.5) --  node [midway,above]{$\ell_{n,m}$}  (3.5,0.5);
\draw[densely dotted] (4.5,0) -- (4.5,0.5);
\draw[densely dotted] (3.5,0) -- (3.5,0.5);

\draw[shift=(right:4.5)] (0pt,3pt) -- (0pt,-3pt) node[below] {$\widehat{F}$};
\draw[shift=(right:3.2)] (0pt,3pt) -- (0pt,-3pt) node[below] {$\mu$};

% \draw[decorate,decoration={brace,raise=0.75cm}] (6,0) -- (4.5,0) ;
% \node at (5.25,-1) {$L$};

% \foreach \x in {0,1,...,9}
%   \draw[shift=(right:\x)] (0pt,3pt) -- (0pt,-3pt) node[below] {$\x$};

% \path (start) to[-bracket  =-80] (2,0);
% \path (7,0)   to[bracket-=-80] (end);
\end{tikzpicture}
\captionof{figure}{$E_L$ true: Event 2}
\endminipage
% }
\end{center}
% \caption{Figure of events when $E_L$ is true}
\end{figure}
\begin{description}
\item[Event 1] If $\mu \ge \widehat{F}$, then $\mu \notin C_n(\alpha;X,W)$ and $E_n(\mu;X,W) = \frac{1}{2}\prod_{t=1}^n (1-W_t(X)(X_t-\mu))\ge E_n(\widehat{F};X,W)  \ge 1/\alpha$ and having that $\mu -(L_Y-\ell_{n,m}) \notin C_m(\alpha;Y,W)$, we have $E_m(\mu -(L_Y-\ell_{n,m});Y,W)=\frac{1}{2}\prod_{t=1}^n\left(1 + W_t(Y)(Y_t-\mu +(L_Y-\ell_{n,m}))\right) \ge 1/\alpha.$
\item[Event 2] If $\mu < \widehat{F}$, then $\mu \notin C_m(\alpha;Y,W)$ and $E_m(\mu;Y,W) = \frac{1}{2}\prod_{t=1}^m (1+W_t(Y)(Y_t-\mu))\ge E_m(\widehat{F};Y,W)  \ge 1/\alpha$ and having that $\mu +(L_X-\ell_{n,m}) \notin C_n(\alpha;X,W)$, we have $E_n(\mu +(L_X-\ell_{n,m});X,W)=\frac{1}{2}\prod_{t=1}^n (1-W_t(X)(X_t-\mu-(L_X-\ell_{n,m}))) \ge 1/\alpha.$
\end{description}
Then, using the bound on the derivative of $E_n$ (Lemma~\ref{lem:bound_deriv}), on Event 1 we have
\begin{align*}
\log(1/\alpha)&\le \log(E_m(\mu - \left(L_Y-\ell_{n,m}\right);Y,W))\\
&= \log\left(\frac{1}{2}\prod_{t=1}^m (1-W_t(Y)(Y_t-\left(L_Y-\ell_{n,m}\right)))\right)\\
&\le \log\left(\frac{1}{2}\prod_{t=1}^m (1+W_t(Y)(Y_t-\mu))\right) + \left(L_Y-\ell_{n,m}\right)\sup_{z} \left| \frac{\d}{\d z}\log\left(\frac{1}{2}\prod_{t=1}^m (1+W_t(Y)(Y_t-z))\right)\right|\\
&\le \log\left(\frac{1}{2}\prod_{t=1}^m (1+W_t(Y)(Y_t-\mu))\right) + \left(L_Y-\ell_{n,m}\right)\sup_{z}\sum_{t=1}^m \frac{W_t(Y)}{1-W_t(Y)(Y_t-z)}\\
&\le \log\left(\frac{1}{2}\prod_{t=1}^m (1+W_t(Y)(Y_t-\mu))\right) + \left(L_Y-\ell_{n,m}\right)\sum_{t=1}^m \frac{W_t(Y)}{1-W_t(Y)(b_Q-a_Q)}
\end{align*} 
and similarly, on Event 2, 
\begin{align*}
\log(1/\alpha)&\le \log\left(\frac{1}{2}\prod_{t=1}^n (1-W_t(X)(X_t-\mu))\right) + \left(L_X-\ell_{n,m}\right)\sum_{t=1}^n \frac{W_t(X)}{1-W_t(X)(b_P-a_P)}.
\end{align*}

Let us take $L_X, L_Y$ such that 
\begin{equation}\label{eq:Lell1}
 L_X -\ell_{n,m} \le  \frac{\eta\log(1/\alpha)}{\sum_{t=1}^n \frac{W_t(X)}{1-W_t(X)(b_P-a_P)} },
\quad \text{and}\ \
 L_Y -\ell_{n,m} \le  \frac{\eta\log(1/\alpha)}{\sum_{t=1}^n \frac{W_t(Y)}{1-W_t(Y)(b_Q-a_Q)} }.
\end{equation}
Then, we have on Event 1,
\begin{align*}
\log(1/\alpha)\le \log\left(\frac{1}{2}\prod_{t=1}^m (1+W_t(Y)(Y_t-\mu))\right) + \eta\log(1/\alpha).
\end{align*}
and on Event 2,
\begin{align*}
\log(1/\alpha)\le \log\left(\frac{1}{2}\prod_{t=1}^n (1-W_t(X)(X_t-\mu))\right) + \eta\log(1/\alpha).
\end{align*}
We get a similar result on Event 2 which implies that the following are true on Event 1 and Event 2 respectively.
\begin{description}
\item[Event 1] If $\mu \ge \widehat{F}$, then $ \frac{1}{2}\prod_{t=1}^n (1-W_t(X)(X_t-\mu))\ge 1/\alpha$ and $\frac{1}{2}\prod_{t=1}^m (1+W_t(Y)(Y_t-\mu)) \ge  e^{-\eta\log(1/\alpha)}/\alpha$.
\item[Event 2] If $\mu < \widehat{F}$, then $ \frac{1}{2}\prod_{t=1}^m (1+W_n(Y)(Y_t-\mu))\ge  1/\alpha$ and $\frac{1}{2}\prod_{t=1}^n (1-W_n(X)(X_t-\mu)))\ge e^{-\eta\log(1/\alpha)}/\alpha$
\end{description}
Then, we can regroup the two cases get that if $E_L$ is true and $C_n(\alpha;X,W) \le C_m(\alpha;Y,W)$, then 
\begin{equation}\label{eq:ville1}
 \frac{1}{4}\prod_{t=1}^n (1-W_t(X)(X_t-\mu))\prod_{t=1}^m (1+W_t(Y)(Y_t-\mu)) \ge e^{-\eta\log(1/\alpha) }/\alpha^2
\end{equation}

Using that $\prod_{i=1}^n (1-W_t(X)(X_i-\mu))\prod_{t=1}^m (1+W_t(Y)(Y_i-\mu))$ is a non-negative martingale with mean $1$, by Ville's inequality, Equation~\eqref{eq:ville1} happens with probability smaller than $\frac{\alpha^2}{4}e^{\eta\log(1/\alpha)} = \frac{1}{4}\alpha^{2-\eta}$.

\paragraph{Case 2:} if $E_L$ is false, then we consider two events

\begin{figure}[h]
\begin{center}
% -*- TeX-master: "article.tex" -*-
% \subfigure[Event 1]{
\minipage{0.45\textwidth}

\begin{tikzpicture}[font=\small, scale=0.8] 
% \draw[-Stealth] (-.5,0) -- (9.5,0) ;
% \foreach \x in {0,1,...,9}
%   \draw[shift=(right:\x)] (0pt,3pt) -- (0pt,-3pt) node[below] {$\x$};

\path (5.5,0) to[bracket-bracket=80] (8,0);
\fill[opacity = 0.2, blue] (5.5,-0.15) -- (8, -0.15) -- (8, 0.15) -- (5.5,0.15) -- cycle;
\node at (6.75,-.5) {$C_\alpha(Y_1^n)$};

\path (1,0)  to[bracket-bracket=80]  (3.5,0) ;
\fill[opacity = 0.2, blue] (1,-0.15) -- (3.5, -0.15) -- (3.5, 0.15) -- (1,0.15) -- cycle;
\node at (2.25,-.5) {$C_\alpha(X_1^n)$};
% \draw (0,1) -- (2,-1);
% \draw (2,-1) -- (4,1);

% \draw[-, ultra thick, blue] (1,0) -- (3.5,0);

\draw[-Stealth] (0.5,0) coordinate (1,0) -- (9,0) coordinate (end);

\draw[Stealth-Stealth] (4.5,1.2) --  node [midway,above]{$L_X$}  (6,1.2);
\draw[densely dotted] (4.5,0) -- (4.5,1.2);
\draw[densely dotted] (6,0) -- (6,1.2);

\draw[Stealth-Stealth] (4.5,0.5) --  node [midway,above]{$\ell_{n,m}$}  (5.5,0.5);
\draw[densely dotted] (4.5,0) -- (4.5,0.5);
\draw[densely dotted] (5.5,0) -- (5.5,0.5);

\draw[shift=(right:4.5)] (0pt,3pt) -- (0pt,-3pt) node[below] {$\widehat{F}$};
\draw[shift=(right:6.5)] (0pt,3pt) -- (0pt,-3pt) node[above, yshift=0.2em] {$\mu$};

% \draw[decorate,decoration={brace,raise=0.75cm}] (6,0) -- (4.5,0) ;
% \node at (5.25,-1) {$L$};

% \foreach \x in {0,1,...,9}
%   \draw[shift=(right:\x)] (0pt,3pt) -- (0pt,-3pt) node[below] {$\x$};

% \path (start) to[-bracket  =-80] (2,0);
% \path (7,0)   to[bracket-=-80] (end);
\end{tikzpicture}
% }
\captionof{figure}{$E_L$ false: Event 1}
\endminipage
\hspace{0.5em}
% \subfigure[Event 2]{
\minipage{0.45\textwidth}

\begin{tikzpicture}[font=\small, scale=0.8] 
% \draw[-Stealth] (-.5,0) -- (9.5,0) ;
% \foreach \x in {0,1,...,9}
%   \draw[shift=(right:\x)] (0pt,3pt) -- (0pt,-3pt) node[below] {$\x$};

\path (5.5,0) to[bracket-bracket=80] (8,0);
\fill[opacity = 0.2, blue] (5.5,-0.15) -- (8, -0.15) -- (8, 0.15) -- (5.5,0.15) -- cycle;
\node at (6.75,-.5) {$C_\alpha(Y_1^n)$};

\path (1,0)  to[bracket-bracket=80]  (3.5,0) ;
\fill[opacity = 0.2, blue] (1,-0.15) -- (3.5, -0.15) -- (3.5, 0.15) -- (1,0.15) -- cycle;
\node at (2.25,-.5) {$C_\alpha(X_1^n)$};
% \draw (0,1) -- (2,-1);
% \draw (2,-1) -- (4,1);

% \draw[-, ultra thick, blue] (1,0) -- (3.5,0);

\draw[-Stealth] (0.5,0) coordinate (1,0) -- (9,0) coordinate (end);

\draw[Stealth-Stealth] (4.5,1.2) --  node [midway,above]{$L_Y$}  (3,1.2);
\draw[densely dotted] (4.5,0) -- (4.5,1.2);
\draw[densely dotted] (3,0) -- (3,1.2);

\draw[Stealth-Stealth] (4.5,0.5) --  node [midway,above]{$\ell_{n,m}$}  (3.5,0.5);
\draw[densely dotted] (4.5,0) -- (4.5,0.5);
\draw[densely dotted] (3.5,0) -- (3.5,0.5);

\draw[shift=(right:4.5)] (0pt,3pt) -- (0pt,-3pt) node[below] {$\widehat{F}$};
\draw[shift=(right:2.5)] (0pt,3pt) -- (0pt,-3pt) node[above, yshift=0.2em] {$\mu$};

% \draw[decorate,decoration={brace,raise=0.75cm}] (6,0) -- (4.5,0) ;
% \node at (5.25,-1) {$L$};

% \foreach \x in {0,1,...,9}
%   \draw[shift=(right:\x)] (0pt,3pt) -- (0pt,-3pt) node[below] {$\x$};

% \path (start) to[-bracket  =-80] (2,0);
% \path (7,0)   to[bracket-=-80] (end);
\end{tikzpicture}
\captionof{figure}{$E_L$ false: Event 2}
\endminipage
% }
\end{center}
\caption{Figure of events when $E_L$ is false}
\end{figure}

\begin{description}
\item[Event 1] If $C_n(\alpha;X,W) \le C_m(\alpha;Y,W)$ and $\mu \ge \widehat{F}+L_X$, then $\mu-(L_X+\ell_{n,m})\ge  C_n(\alpha;X,W)$.
\item[Event 2] If $C_n(\alpha;X,W) \le C_m(\alpha;Y,W)$ and $\mu < \widehat{F}-L_Y$, then $\mu+(L_Y+\ell_{n,m}) \le  C_m(\alpha;Y,W)$.
\end{description}
On Event 1, we have, using Lemma~\ref{lem:convex} and Lemma~\ref{lem:bound_deriv} and the fact that $\mu-(L_X+\ell_{n,m})\ge  C_n(\alpha;X,W)$,
\begin{align*}
\log(1/\alpha)&\le \log(E_n(\mu-(L_X+\ell_{n,m});X,W))\\
&= \log\left(\frac{1}{2}\prod_{t=1}^n (1-W_t(X)(X_t-(\mu-(L_X+\ell_{n,m}))))\right)\\
&\le \log\left(\frac{1}{2}\prod_{t=1}^n (1-W_t(X)(X_t-\mu))\right) - \left(\ell_{n,m}+L_X\right) \inf_{z}\left| \frac{\d}{\d z}\log\left(\frac{1}{2}\prod_{t=1}^n (1-W_t(X)(X_t-z))\right)\right|\\
&\le  \log\left(\frac{1}{2}\prod_{t=1}^n (1-W_t(X)(X_t-\mu))\right) -\left(\ell_{n,m}+L_X\right)\sum_{t=1}^n \frac{W_t}{1+W_t(b_P-a_P)}
\end{align*} 
and similarly on Event 2. Hence, taking
\begin{equation}\label{eq:Lell2}
L_X +\ell_{n,m} \ge   \frac{\eta'\log(1/\alpha)}{\sum_{t=1}^n \frac{W_t(X)}{1+W_t(X)(b_P-a_P)} },
\quad\text{and}\ \ 
L_Y +\ell_{n,m} \ge   \frac{\eta'\log(1/\alpha)}{\sum_{t=1}^m \frac{W_t(Y)}{1+W_t(Y)(b_Q-a_Q)} },
\end{equation}
we get on Event $1$,
$$\log(1/\alpha) \le \log\left(\frac{1}{2}\prod_{t=1}^n (1-W_t(X)(X_t-\mu))\right) -\eta' \log(1/\alpha) $$
and similarly for Event 2,
$$\log(1/\alpha) \le \log\left(\frac{1}{2}\prod_{t=1}^m (1+W_t(Y)(Y_t-\mu))\right) -\eta' \log(1/\alpha) $$
Hence, 
\begin{description}
\item[Event 1] If $C_n(\alpha;X,W) \le C_m(\alpha;Y,W)$ and $\mu \ge \widehat{F}+L$, then $\frac{1}{2}\prod_{t=1}^n (1-W_t(X)(X_t-\mu))\ge  1/\alpha^{\eta'+1}$.
\item[Event 2] If $C_n(\alpha;X,W) \le C_m(\alpha;Y,W)$ and $\mu < \widehat{F}-L$, then $\frac{1}{2}\prod_{t=1}^m (1+W_t(Y)(Y_t-\mu))\ge  1/\alpha^{\eta'+1}$.
\end{description}
Then, by Ville's inequality, Event 1 or Event 2 happen with probability smaller than $\alpha^{\eta'+1}$. Equations~\eqref{eq:Lell1} and \eqref{eq:Lell2} together lead us to choose
$$L_X = \frac{\log(1/\alpha)}{2} \left(\frac{\eta'}{\sum_{t=1}^n \frac{W_t(X)}{1+W_t(X)(b_P-a_P)} }+\frac{\eta}{\sum_{t=1}^n \frac{W_t(X)}{1-W_t(X)(b_P-a_P)} }  \right),$$
$$L_Y = \frac{\log(1/\alpha)}{2} \left(\frac{\eta'}{\sum_{t=1}^m \frac{W_t(Y)}{1+W_t(Y)(b_Q-a_Q)} }+\frac{\eta}{\sum_{t=1}^m \frac{W_t(Y)}{1-W_t(Y)(b_Q-a_Q)} }  \right)$$
and
\begin{align*}
\ell_{n,m} = \frac{\log(1/\alpha)}{2}\max&\left(\frac{\eta'}{\sum_{t=1}^n \frac{W_t(X)}{1+W_t(X)(b_P-a_P)} }-\frac{\eta}{\sum_{t=1}^n \frac{W_t(X)}{1-W_t(X)(b_P-a_P)}} \right.\\
&\phantom{=}\left., \frac{\eta'}{\sum_{t=1}^m \frac{W_t(Y)}{1+W_t(Y)(b_Q-a_Q)} }-\frac{\eta}{\sum_{t=1}^m \frac{W_t(Y)}{1-W_t(Y)(b_Q-a_Q)}} \right).
\end{align*}

Remark finally that $\ell_{n,m}\mapsto \P\left(\exists n,m\ge1 : \ C_n(\alpha;X,W)\le   C_m(\alpha;Y,W) - \ell_{n,m} \right)$ is non-increasing which means that the condition on $\ell_{n,m}$ can be an inequality instead of an equality. 

The bound on deciding $H_1^-$ when $H_0$ was true is exactly symmetric and give the same probability of error.

\end{document}